\documentclass[11pt]{article}
\usepackage[margin=1in]{geometry} 
\usepackage{amssymb,amsfonts,amsmath,bbm,mathrsfs,stmaryrd,mathtools,mathabx}
\usepackage{xcolor}
\usepackage{url}

\usepackage{extarrows}

\usepackage[shortlabels]{enumitem}
\usepackage{tensor}

\usepackage{xr}
\usepackage[T1]{fontenc}

\usepackage[colorlinks,
linkcolor=black!75!red,
citecolor=blue,
pdftitle={},
pdfproducer={pdfLaTeX},
pdfpagemode=None,
bookmarksopen=true,
bookmarksnumbered=true,
backref=page]{hyperref}

\usepackage{tikz}
\usetikzlibrary{arrows,calc,decorations.pathreplacing,decorations.markings,intersections,shapes.geometric,through,fit,shapes.symbols,positioning,decorations.pathmorphing}

\usepackage{braket}

\usepackage[amsmath,thmmarks,hyperref]{ntheorem}
\usepackage{cleveref}

\newcommand\nthalias[1]{\AddToHook{env/#1/begin}{\crefalias{lemma}{#1}}}

\nthalias{definition}
\nthalias{example}
\nthalias{examples}
\nthalias{remark}
\nthalias{remarks}
\nthalias{convention}
\nthalias{notation}
\nthalias{construction}
\nthalias{theoremN}
\nthalias{propositionN}
\nthalias{corollaryN}
\nthalias{lemma}
\nthalias{proposition}
\nthalias{corollary}
\nthalias{theorem}
\nthalias{conjecture}
\nthalias{question}
\nthalias{assumption}

\creflabelformat{enumi}{#2#1#3}

\crefname{section}{Section}{Sections}
\crefformat{section}{#2Section~#1#3} 
\Crefformat{section}{#2Section~#1#3} 

\crefname{subsection}{\S}{\S\S}
\AtBeginDocument{%
  \crefformat{subsection}{#2\S#1#3}%
  \Crefformat{subsection}{#2\S#1#3}%
}

\crefname{subsubsection}{\S}{\S\S}
\AtBeginDocument{%
  \crefformat{subsubsection}{#2\S#1#3}%
  \Crefformat{subsubsection}{#2\S#1#3}%
}

\theoremstyle{plain}

\newtheorem{lemma}{Lemma}[section]
\newtheorem{proposition}[lemma]{Proposition}
\newtheorem{corollary}[lemma]{Corollary}
\newtheorem{theorem}[lemma]{Theorem}

\theoremstyle{plain}
\theoremnumbering{Alph}
\newtheorem{theoremN}{Theorem}

\theoremstyle{plain}
\theorembodyfont{\upshape}
\theoremsymbol{\ensuremath{\blacklozenge}}

\newtheorem{definition}[lemma]{Definition}

\newtheorem{examples}[lemma]{Examples}
\newtheorem{remark}[lemma]{Remark}
\newtheorem{remarks}[lemma]{Remarks}

\newtheorem{notation}[lemma]{Notation}

\crefname{definition}{definition}{definitions}
\crefformat{definition}{#2definition~#1#3} 
\Crefformat{definition}{#2Definition~#1#3} 

\crefname{ex}{example}{examples}
\crefformat{example}{#2example~#1#3} 
\Crefformat{example}{#2Example~#1#3} 

\crefname{exs}{example}{examples}
\crefformat{examples}{#2example~#1#3} 
\Crefformat{examples}{#2Example~#1#3} 

\crefname{remark}{remark}{remarks}
\crefformat{remark}{#2remark~#1#3} 
\Crefformat{remark}{#2Remark~#1#3} 

\crefname{remarks}{remark}{remarks}
\crefformat{remarks}{#2remark~#1#3} 
\Crefformat{remarks}{#2Remark~#1#3} 

\crefname{convention}{convention}{conventions}
\crefformat{convention}{#2convention~#1#3} 
\Crefformat{convention}{#2Convention~#1#3} 

\crefname{notation}{notation}{notations}
\crefformat{notation}{#2notation~#1#3} 
\Crefformat{notation}{#2Notation~#1#3} 

\crefname{table}{table}{tables}
\crefformat{table}{#2table~#1#3} 
\Crefformat{table}{#2Table~#1#3}

\crefname{lemma}{lemma}{lemmas}
\crefformat{lemma}{#2lemma~#1#3} 
\Crefformat{lemma}{#2Lemma~#1#3} 

\crefname{proposition}{proposition}{propositions}
\crefformat{proposition}{#2proposition~#1#3} 
\Crefformat{proposition}{#2Proposition~#1#3} 

\crefname{propositionN}{proposition}{propositions}
\crefformat{propositionN}{#2proposition~#1#3} 
\Crefformat{propositionN}{#2Proposition~#1#3} 

\crefname{corollary}{corollary}{corollaries}
\crefformat{corollary}{#2corollary~#1#3} 
\Crefformat{corollary}{#2Corollary~#1#3} 

\crefname{corollaryN}{corollary}{corollaries}
\crefformat{corollaryN}{#2corollary~#1#3} 
\Crefformat{corollaryN}{#2Corollary~#1#3} 

\crefname{theorem}{theorem}{theorems}
\crefformat{theorem}{#2theorem~#1#3} 
\Crefformat{theorem}{#2Theorem~#1#3} 

\crefname{theoremN}{theorem}{theorems}
\crefformat{theoremN}{#2theorem~#1#3} 
\Crefformat{theoremN}{#2Theorem~#1#3} 

\crefname{enumi}{}{}
\crefformat{enumi}{#2#1#3}
\Crefformat{enumi}{#2#1#3}

\crefname{assumption}{assumption}{Assumptions}
\crefformat{assumption}{#2assumption~#1#3} 
\Crefformat{assumption}{#2Assumption~#1#3} 

\crefname{construction}{construction}{Constructions}
\crefformat{construction}{#2construction~#1#3} 
\Crefformat{construction}{#2Construction~#1#3} 

\crefname{equation}{}{}
\crefformat{equation}{(#2#1#3)} 
\Crefformat{equation}{(#2#1#3)}

\numberwithin{equation}{section}

\theoremstyle{nonumberplain}
\theoremsymbol{\ensuremath{\blacksquare}}

\newtheorem{proof}{Proof}

\newcommand\bG{{\mathbb G}}
\newcommand\bH{{\mathbb H}}

\newcommand\bL{{\mathbb L}}
\newcommand\bM{{\mathbb M}}

\newcommand\bQ{{\mathbb Q}}

\newcommand\cC{{\mathcal C}}
\newcommand\cD{{\mathcal D}}

\newcommand\cF{{\mathcal F}}

\newcommand\cJ{{\mathcal J}}

\newcommand\cS{{\mathcal S}}

\newcommand\cU{{\mathcal U}}
\newcommand\cV{{\mathcal V}}

\DeclareMathOperator{\id}{id}

\newcommand{\cat}[1]{\textsc{#1}}

\title{Monadic functors forgetful of (dis)inhibited actions}
\author{Alexandru Chirvasitu}

\begin{document}

\date{}

\newcommand{\Addresses}{{
  \bigskip
  \footnotesize

  \textsc{Department of Mathematics, University at Buffalo}
  \par\nopagebreak
  \textsc{Buffalo, NY 14260-2900, USA}
  \par\nopagebreak
  \textit{E-mail address}: \texttt{achirvas@buffalo.edu}


}}

\maketitle

\begin{abstract}
  We prove a number of results of the following common flavor: for a category $\mathcal{C}$ of topological or uniform spaces with all manner of other properties of common interest (separation / completeness / compactness axioms), a group (or monoid) $\mathbb{G}$ equipped with various types of topological structure (topologies, uniformities) and the corresponding category $\mathcal{C}^{\mathbb{G}}$ of appropriately compatible $\mathbb{G}$-flows in $\mathcal{C}$, the forgetful functor $\mathcal{C}^{\mathbb{G}}\to \mathcal{C}$ is monadic. In all cases of interest the domain category $\mathcal{C}^{\mathbb{G}}$ is also cocomplete, so that results on adjunction lifts along monadic functors apply to provide equivariant completion and/or compactification functors. This recovers, unifies and generalizes a number of such results in the literature due to de Vries, Mart'yanov and others on existence of equivariant compactifications / completions and cocompleteness of flow categories.
\end{abstract}


\vspace{.5cm}



\section*{Introduction}

The general theme underlying the sequel is that of equivariant topologically-flavored structures: topologies, quasi-topologies, uniformities and the like, and their behavior in the presence of an action by a group (or more generally monoid) $\bG$. We refer to such a structure as a {\it $\bG$-flow} in the relevant category (of topological spaces, etc.: \Cref{def:gflow} makes this precise), with the understanding that the unqualified term does not entail any default continuity assumptions on the map $\bG\times X\to X$ implementing the flow.

Universal compactifications of $\bG$-flows offer part of the motivation. Recall \cite[\S 4.4.4]{dvr_ttg} that for any topological group $\bG$, the inclusion functor
\begin{equation}\label{eq:cpctg2topg}
  \text{continuous compact Hausdorff $\bG$-flows}
  =:
  \tensor*[_\iota]{\cat{Cpct}}{_{T_2}^{\bG}}
  \quad
  \lhook\joinrel\xrightarrow{\quad}
  \quad
  \tensor*[_\iota]{\cat{Top}}{^{\bG}}
  :=
  \text{continuous $\bG$-flows}
\end{equation}
has a left adjoint (\Cref{not:flowtypes} explains the left-hand `$\iota$' subscripts). In other words, the full left-hand subcategory is {\it reflective} \cite[Definition 3.5.2]{brcx_hndbk-1}), associating to $\bG\times X\to X$ the familiar {\it universal $\bG$-equivariant compactification} $\beta_{\bG} X$ of $X$ (\cite[\S 1]{ilm_max-equiv-cpct}, \cite{bh_flows,zbMATH06679171,megr_max-equiv-cpct}, their many references, etc.). The inclusion \Cref{eq:cpctg2topg} and its reflection moreover fit into a richer picture: a $\bG$-action $\bG\times X\xrightarrow{\triangleright}X$ on $X$ equips the latter with a {\it uniform structure} \cite[Definition 7.1]{james_unif_1999} $(X,\cU_{\triangleright})$ defined as the finest among those satisfying the following requirements:
\begin{itemize}[wide]
\item its induced topology on $X$ is coarser than the original one;

\item the uniformity is compatible with the action in the sense that elements of $\bG$ send entourages to entourages;

\item and the action is in addition {\it bounded} with respect to the uniformity, in the sense of \cite[\S 2, p.276]{dvr_ex}:
  \begin{equation}\label{eq:eunif}
    \forall\text{ entourage }V\subseteq Y^2,\quad\exists\text{ nbhd }N\ni 1\in \bG
    \quad\text{with}\quad
    \{(s\triangleright y,y)\ |\ s\in N,\ y\in Y\}\subseteq V. 
  \end{equation}
\end{itemize}


The left adjoint $\beta_{\bG}$ to \Cref{eq:cpctg2topg} then factors through the category $\cat{unif}$ of uniform spaces (with {\it uniformly continuous} maps \cite[Definition 7.7]{james_unif_1999} as morphisms). The embedding
\begin{equation*}
  \cat{Cpct}_{T_2}
  \lhook\joinrel\xrightarrow{\quad}
  \cat{Unif}
\end{equation*}
obtained by equipping every compact Hausdorff space with its unique uniformity \cite[Proposition 8.20]{james_unif_1999} compatible with its topology also has a left adjoint
\begin{equation*}
  \cat{Unif}\ni
  (X,\cU)
  \xmapsto{\quad\beta_{\bullet}\quad}
  \beta_{\cU}X
  \in \cat{Cpct}_{T_2},
\end{equation*}
assigning to a uniform space $(X,\cU)$ its {\it Samuel compactification} $\beta_{\cU}X$ of \cite[Theorem II.32]{isb_unif_1964}. The equivariant compactification $\beta_{\bG}X$ is then nothing but $\beta_{\cU_{\triangleright}}X$, equipped with the natural $\bG$-action the latter inherits from $X$. Thus:

\begin{equation}\label{eq:unifgdiag}
  \begin{tikzpicture}[>=stealth,auto,baseline=(current  bounding  box.center)]
    \path[anchor=base] 
    (0,0) node (l) {$\tensor*[_\iota]{\cat{Top}}{^{\bG}}$}
    +(3,.5) node (ul) {$\tensor[_b]{\cat{Unif}}{^{\bG}}$}
    +(10,0) node (ur) {$\tensor*[_\iota]{\cat{Cpct}}{_{T_2}^{\bG}}$}
    +(5,0) node (dl) {$\cat{Unif}$}
    +(7,-.5) node (dr) {$\cat{Cpct}_{T_2}$}
    ;
    \draw[->] (l) to[bend left=6] node[pos=.5,auto] {$\scriptstyle $} (ul);
    \draw[->] (ul) to[bend left=6] node[pos=.5,auto] {$\scriptstyle $} (dl);
    \draw[->] (ur) to[bend left=6] node[pos=.5,auto] {$\scriptstyle $} (dr);
    \draw[->] (l) to[bend right=6] node[pos=.5,auto,swap] {$\scriptstyle \triangleright\mapsto \cU_{\triangleright}$} (dl);
    \draw[->] (dl) to[bend right=6] node[pos=.2,auto,swap] {$\scriptstyle \beta_{\bullet}$} (dr);
    \draw[->] (l) to[bend right=30] node[pos=.5,auto] {$\scriptstyle \beta_{\bG}$} (ur);
  \end{tikzpicture}
\end{equation}
with forgetful unmarked downward arrows. The symbol $\tensor[_b]{\cat{Unif}}{^{\bG}}$ stands for the category of uniform spaces $(Y,\cU)$ equipped with bounded $\bG$-flows $\bG\times Y\xrightarrow{\triangleright}Y$ (the $\cat{EUnif}^{\bG}$ of \cite[Definition 3.2(2)]{megr_max-equiv-cpct}, etc.). Indeed, to verify that $\beta_{\bG}X\cong \beta_{\cU_{\triangleright}}X$, note that
\begin{itemize}[wide]
\item the action of $\bG$ travels to a continuous one on $X$ topologized with its induced $\cU_{\triangleright}$-topology by the assumed coarseness of that topology;

\item thence also to a continuous action on the Samuel compactification $\beta_{\cU_{\triangleright}}X$ by \cite[Proposition 2.2]{ilm_max-equiv-cpct};
  
\item on the one hand the pullback of the unique \cite[Theorem 36.19]{wil_top} compatible uniformity on $\beta_{\bG}X$ along $X\to \beta_{\bG}X$ has the required properties, so the universality of $\beta_{\cU_{\triangleright}}$ provides an equivariant map $\beta_{\cU_{\triangleright}}X\to \beta_{\bG}X$;

\item and conversely, the universality of $\beta_{\bG}$ ensures the existence of an inverse $\beta_{\bG}X\to \beta_{\cU_{\triangleright}}X$ for the map in the preceding item.
\end{itemize}


Given that the left adjoint to \Cref{eq:cpctg2topg} is a $\bG$-equivariant version of the much more familiar {\it Stone-\v{C}ech compactification} \cite[Example 3.3.9.c]{brcx_hndbk-1}, it seems reasonable to fit such left adjunctions into a broader framework whereby the $\bG$-actions ``come along for the ride''. Formally, the observation is that in all instances discussed above (and more), equivariant and ``absolute'' or ``plain'' compactifications are related through {\it adjunction lifting} \cite[\S 4.5]{brcx_hndbk-2} along {\it monadic} functors \cite[Definition 4.4.1]{brcx_hndbk-2}; we elaborate below, after a brief reminder (\cite[\S\S 4.1, 4.2]{brcx_hndbk-2} or \cite[\S\S 3.1, 3.2]{bw} or \cite[\S\S VI.1-3]{mcl_2e} for the standard theory, \cite[\S II]{dub_kan-enr} for the enriched-category version, and so on).

\begin{itemize}[wide]
\item A {\it monad} (or {\it triple}) on a category $\cC$ is an endofunctor $\cC\xrightarrow{T}\cC$ equipped with natural transformations
  \begin{equation*}
    T\circ T\xrightarrow[\text{associative}]{\mu}T
    \quad\text{and}\quad
    \id\xrightarrow[\text{unital with respect to $\mu$}]{\eta} T.
  \end{equation*}
  In short: a monoid in the monoidal category of endofunctors of $\cC$ with composition for its monoidal structure.

\item An {\it algebra} over a monad $T$ (or {\it $T$-algebra}) is an object $X\in \cC$ equipped with a morphism $TX\to X$ appropriately associative and unital with respect to $\mu$ and $\eta$.

  $T$-algebras form {\it Eilenberg-Moore category} $\cC^T$ of {\it $T$-algebras}, equipped with a functor $\cC^T\xrightarrow{\cat{fgt}} \cC$ forgetting the algebra structure maps $TX\to X$: \cite[Definition 4.1.2]{brcx_hndbk-2} for plain categories, or \cite[\S II.1, preceding Proposition II.1.1]{dub_kan-enr} for the enriched version.

\item A functor $\cC'\to \cC$ is {\it monadic} (or {\it tripleable}) if it fits into a diagram
  \begin{equation*}
    \begin{tikzpicture}[>=stealth,auto,baseline=(current  bounding  box.center)]
      \path[anchor=base] 
      (0,0) node (l) {$\cC'$}
      +(2,-.5) node (d) {$\cC$}
      +(4,0) node (r) {$\cC^T$,}
      ;
      \draw[->] (l) to[bend left=6] node[pos=.5,auto] {$\scriptstyle \simeq$} (r);
      \draw[->] (l) to[bend right=6] node[pos=.5,auto,swap] {$\scriptstyle $} (d);
      \draw[->] (r) to[bend left=6] node[pos=.5,auto] {$\scriptstyle \cat{fgt}$} (d);
    \end{tikzpicture}
  \end{equation*}
  commutative up to natural isomorphism. 
\end{itemize}

The point now is that each square in the commutative functor diagram
\begin{equation}\label{eq:2commsq}
  \begin{tikzpicture}[>=stealth,auto,baseline=(current  bounding  box.center)]
    \path[anchor=base] 
    (0,0) node (lu) {$\tensor*[_\iota]{\cat{Cpct}}{_{T_2}^{\bG}}$}
    +(2,0) node (mu) {$\tensor[_b]{\cat{Unif}}{^{\bG}}$}
    +(4,0) node (ru) {$\tensor*[_\iota]{\cat{Top}}{^{\bG}}$}
    +(1,-1) node (ld) {$\cat{Cpct}_{T_2}$}
    +(3,-1) node (md) {$\cat{Unif}$}
    +(5,-1) node (rd) {$\cat{Top}$}
    ;
    \draw[right hook->] (lu) to[bend left=6] node[pos=.5,auto] {$\scriptstyle $} (mu);
    \draw[->] (mu) to[bend left=6] node[pos=.5,auto] {$\scriptstyle $} (ru);
    \draw[right hook->] (ld) to[bend right=6] node[pos=.5,auto] {$\scriptstyle $} (md);
    \draw[->] (md) to[bend right=6] node[pos=.5,auto] {$\scriptstyle $} (rd);
    \draw[->] (lu) to[bend right=6] node[pos=.5,auto] {$\scriptstyle $} (ld);
    \draw[->] (mu) to[bend left=6] node[pos=.5,auto] {$\scriptstyle $} (md);
    \draw[->] (ru) to[bend left=6] node[pos=.5,auto] {$\scriptstyle $} (rd);
  \end{tikzpicture}
\end{equation}
with forgetful downward arrows, and analogous squares involving ``interpolating'' categories such as $\cat{Top}_{T_2}$ (Hausdorff spaces) and $\cat{Top}_{T_{3\frac 12}}$ (Tychonoff spaces), fits into the framework of the adjunction lifting theorem \cite[Theorem 4.5.6 and Exercise 4.8.5]{brcx_hndbk-2}: the downward arrows are monadic and the top categories have appropriate colimits, and hence the top horizontal functors have left adjoints as soon as the bottom ones do. A heavily abbreviated sampling of \Cref{th:monads,th:allcocompl,cor:liftreflect}, then, reads as follows. 

\begin{theoremN}\label{thn:main}
  \begin{enumerate}[(1),wide]
  \item\label{item:thn:main:monads} For every topological group $\bG$ the forgetful functors $\tensor[_\cdot]{\cC}{^{\bG}}\to \cC$ are all monadic, with $\cdot\in\{\iota,b\}$ as appropriate and $\cC$ ranging over any of the categories
    \begin{itemize}[-,wide]
    \item $\cat{Top}_{\bullet}$ with $\bullet\in \{\text{blank},\ T_2,\ T_{2f}=\text{functionally Hausdorff},\ T_{3\frac 12}\}$;

    \item or $\cat{Unif}_{\bullet}$ with $\bullet\in\left\{\text{blank},\ T_2=\text{Hausdorff},\ (T_2,c)=\text{complete Hausdorff}\right\}$;

    \item or $\cat{Cpct}_{T_2}$. 
    \end{itemize}

  \item\label{item:thn:main:cocompl} The categories $\tensor[_\cdot]{\cC}{^{\bG}}$ of \Cref{item:thn:main:monads} are all also cocomplete

  \item\label{item:thn:main:monlift} Consequently, for any of the reflective inclusion functors $\cC\lhook\joinrel\to \cD$ the corresponding $\tensor[_\cdot]{\cC}{^{\bG}}\lhook\joinrel\to \tensor[_\cdot]{\cD}{^{\bG}}$ is also reflective by monadic left-adjoint lifts. 
  \end{enumerate}
\end{theoremN}

Recall \cite[Definition 12.2]{ahs} that {\it cocomplete} categories $\cC$ are those for which all functors $\cD\to \cC$ with small $\cD$ have colimits. Equivalently \cite[Theorem 12.3]{ahs}, $\cC$ admits coequalizers for parallel pairs of arrows and coproducts of arbitrary families of objects. 

Offshoots of this main thread include

\begin{itemize}[wide]
\item a generalization (\Cref{cor:cdgengtychcoco}) of the main result of \cite{marty_cocompl} to the effect that for a Hausdorff topological group $\bG$ the category of {\it $\bG$-Tychonoff} flows (i.e. \cite[p.220]{megr_g-cats} those embedding homeomorphically onto their image in the $\bG$-equivariant compactification) is cocomplete;

\item left adjoints of which the construction $(X,\cU)\mapsto (X,\cU^{\bG})$ of \cite[Lemma 3.8]{megr_max-equiv-cpct}, universally attaching a bounded flow in $\cat{Unif}$ to a {\it quasi-bounded} one \cite[Definition 3.2(4)]{megr_max-equiv-cpct}, is a particular case.   
\end{itemize}

\subsection*{Acknowledgements}

This work is partially supported through NSF grant DMS-2001128. 

I am grateful for helpful comments, questions, and observations from G. Luk\'acs and M. Megrelishvili and last but not least, for the anonymous referee's interest, patience and tenacity. 

\section{Preliminaries}\label{se:prel}

Some commonly-employed notation and terminology:

\begin{itemize}[wide]
\item The hom space of morphisms $X\to Y$ in a category $\cC$ is $\cC(X,Y)$. On the few occasions when they come up, opposite categories carry a `$\circ$' superscript (as in $\cC^{\circ}$). 
  
\item $\cat{Set}$, $\cat{Top}$, $\cat{Cpct}$ and $\cat{Unif}$ denote the categories of sets and topological, compact and {\it uniform} spaces respectively; for the latter we refer the reader to \cite[Chapter 7]{james_unif_1999}, \cite[Chapter II]{bourb_top_en_1}, etc., with more specific citations below, as needed.
  
\item We will often speak of {\it $\cD$-concrete categories} $(\cC,U)$, i.e. \cite[Definition 5.1]{ahs} faithful functors $\cC\xrightarrow{U}\cD$. $\cat{Set}$-concrete categories (the {\it constructs} of \cite[Definition 5.1(2)]{ahs}) are just plain {\it concrete}. 
  
\item {\it Separation axioms} (\cite[\S 13 and \S 35]{wil_top} for topologies and uniformities respectively) occasionally decorate the main category symbols as subscripts: $\cat{Top}_{T_2}$ and $\cat{Unif}_{T_2}$ for {\it Hausdorff} topological and uniform spaces respectively, for instance, $\cat{Top}_{T_{3\frac 12}}$ for {\it Tychonoff} \cite[Definition 14.8]{wil_top} (or Hausdorff {\it completely regular}) spaces, etc. Another example of occasional interest (for instance through its relevance to operator algebras \cite[Definition 2.2]{phil_inv-lim}) is the category $\cat{Top}_{T_{2f}}$ of {\it functionally Hausdorff} spaces \cite[Problem 14G]{wil_top}: those admitting continuous real-valued functions assigning any two distinct points distinct values. 

\item $\cat{Top}_{\kappa}$ is the category of {\it compactly generated} spaces (or {\it $\kappa$-spaces}), i.e. \cite[Definition 43.8]{wil_top} the spaces whose open sets are precisely those whose intersection with every compact subspace is open (equivalently: carrying the {\it final topology} \cite[\S I.2.4, Proposition 6]{bourb_top_en_1} induced by the inclusions of its compact subspaces).
    
  $\cat{Top}_{\kappa}$ is coreflective in $\cat{Top}$ (\cite[\S VII.8, Proposition 2]{mcl_2e} for the Hausdorff version $\cat{Top}_{T_2,\kappa}\subset \cat{Top}_{T_2}$), so in particular (co)complete. By \cite[\S VII.8, Theorem 3]{mcl_2e} (and its non-Hausdorff counterpart) $\cat{Top}_{\kappa}$ and $\cat{Top}_{T_2,\kappa}$ are also {\it Cartesian closed} \cite[\S IV.6]{mcl_2e} for their product $\times^{\kappa}$ (henceforth the {\it $\kappa$-product}) obtained by composing the usual Cartesian product with the coreflection $\cat{Top}_{T_2}\to \cat{Top}_{T_2,\kappa}$: all endofunctors $-\times^{\kappa} X$ are left adjoints. 

\item We also write $\cat{Unif}_{T_2,c}$ for the category of {\it complete} \cite[\S II.3.3, Definition 3]{bourb_top_en_1} Hausdorff uniform spaces (completeness makes sense without separation, but is better behaved categorically in its presence). For equivariant {\it completions} (rather than compactifications) the reader can consult, say, \cite{zbMATH00690930} and its sources. 
 
\item For a {\it monoidal category} \cite[Definition 6.1.1]{brcx_hndbk-2} $(\cC,\otimes,1_{\cC})$ we write $\cat{Gr}(\cC)$ or $\cat{Mon}(\cC)$ for the categories of groups or respectively monoids {\it internal} to it: objects $X\in \cC$ equipped with associative morphisms $X\otimes X\to X$ and units $1_{\cC}\to X$, along also with an inversion $X\xrightarrow{(-)^{-1}}X$ in the case of $\cat{Gr}$, all mutually compatible in the familiar sense (see e.g. \cite[\S III.6]{mcl_2e} for {\it Cartesian} monoidal categories, i.e. those with finite products and $\otimes=\times$).
\end{itemize}

Whenever an object $Y$ in a monoidal category $(\cV,\otimes,1_{\cV})$ is {\it exponentiable} in the sense \cite[Definition 7.1.3]{brcx_hndbk-2} that $\cV\xrightarrow{-\otimes Y}\cV$ is left adjoint to a functor $[Y,-]$, there is a correspondence
\begin{equation*}
  \begin{tikzpicture}[>=stealth,auto,baseline=(current  bounding  box.center)]
    \path[anchor=base] 
    (0,0) node (l) {$\bigg(\text{morphisms }X\otimes Y\to Z\bigg)$}
    +(5,0) node (m) {$\cong$}
    +(10,0) node (r) {$\bigg(\text{morphisms }X \to [Y,Z]\bigg)$}
    ;
    \draw[->] (l) to[bend left=10] node[pos=.5,auto] {$\scriptstyle \text{currying}$} (r);
    \draw[->] (r) to[bend left=10] node[pos=.5,auto] {$\scriptstyle \text{uncurrying}$} (l);
  \end{tikzpicture}
\end{equation*}
(in terminology well familiar to theoretical computer scientists \cite[\S 5.1]{sk_form-synt-sem} and also occasionally in use in category theory \cite[Definition 14]{bs_rosetta}). We frequently (and sometimes tacitly) take this for granted, often for set maps, with $\otimes=\times$ and $[X,Y]=$ functions $X\to Y$. 


\section{Monadic compactification / completion lifts}\label{se:qtop}

The central objects under consideration are {\it flows} in categories.

\begin{definition}\label{def:gflow}
  For a monoid $\bM$ a {\it flow on} an object $X\in \cC$ of a category $\cC$ is a monoid morphism $\bM\to \cC(X,X)$. 
\end{definition}

\begin{remarks}\label{res:postflow}
  \begin{enumerate}[(1),wide]
  \item\label{item:res:postflow:whyflow} The term `flow' is in wide use in the literature (e.g. \cite{zbMATH03595892}), and its advantage over `action' is that it seems somewhat more natural to transport attributes of the underlying space $X$ to the former (rather than the latter): compact (Hausdorff) flows are those for which $X$ is compact (Hausdorff), similarly for Tychonoff spaces/flows, etc.
   
  \item\label{item:res:postflow:allconcrete} The categories of interest in the present section are all concrete, and hence flows can always be interpreted as just plain uncurried set maps $\bG\times X\xrightarrow{\triangleright} X$ (unital and associative, as usual). Even when $\bG$ is topological and $\cC$ is some category of topological spaces, though, it is occasionally convenient to consider flows whose underlying map $\triangleright$ is not necessarily continuous. We allow for this by further qualifying the flow:
    \begin{itemize}[wide]
    \item If $\bG$ is a topological group then a flow on $X$ in a $\cat{Top}$-concrete category is continuous if the map $\bG\times X\xrightarrow{\triangleright} X$ is.
      
    \item Under the same circumstances the flow is only {\it separately} continuous if $s\triangleright x$, $s\in \bG$, $x\in X$ is continuous in each variable if the other is kept fixed, etc.
    \end{itemize}
  \end{enumerate}  
\end{remarks}

\begin{notation}\label{not:flowtypes}
  For monoidal $\cC$ and internal monoids $\bM\in\cat{Mon}(\cC)$ one can also consider the category $\tensor[_\iota]{\cC}{^{\bM}}$ of objects $X\in \cC$ equipped with appropriately unital associative morphisms $\bM\otimes X\to X$ in $\cC$ (the left-hand subscript stands for `internal'). We apply this to subcategories $\cD\subseteq \cC$ as well, writing $\tensor[_\iota]{\cD}{^{\bM}}$ for the internal flows $\bM\otimes X\to X$ with $X\in \cD$ (even when $\bM$ itself is not an object of $\cD$, but only of the larger $\cC$).

  $\tensor[_\iota]{\cat{Top}}{^{\bG}}$, for instance, is the category of continuous flows for topological groups $\bG$. In practice, such internal actions will always also be flows in the sense of \Cref{def:gflow} by (un)currying, so we refer to them as such. Other left-hand decorations occasionally appear, as in \Cref{eq:unifgdiag}.

  For more general families $\cF$ of conditions we might demand flows satisfy we employ the generic symbol $\tensor[_\cF]{\cC}{^{\bM}}$ for the category of flows in $\cC$ meeting those requirements.
\end{notation}

\begin{remark}\label{res:whyactextends}
  


  The boundedness of \Cref{eq:eunif} can be phrased in the spirit of $\tensor[_\iota]{\cat{Unif}}{^{\bG}}$, as the requirement that the action $\bG\times Y\xrightarrow{\triangleright}Y$ be uniformly continuous for appropriate uniform structures on its (co)domain:

  \begin{itemize}[wide]
  \item equip the left-hand $Y$ with the {\it discrete uniformity} \cite[Examples 0.6]{rd_unif-gps}, consisting of all subsets of $Y\times Y$ containing the diagonal;

  \item $\bG$ with its {\it right uniformity} \cite[Lemma-Definition 2.1]{rd_unif-gps}, with entourages consisting of $(s',s)\in \bG^2$ with $s's^{-1}$ close to $1\in \bG$;
    
  \item the product $\bG\times Y$ with the resulting  {\it product uniformity } (\cite[Example 0.20(b)]{rd_unif-gps} of the two;

  \item and the right-hand (codomain) $Y$ with its original uniformity. 
  \end{itemize}
  Indeed, uniform continuity for said structures translates to
  \begin{equation*}
    \forall\text{ entourage }V\subseteq Y^2,\quad\exists\text{ nbhd }N\ni 1\in \bG
    \quad\text{with}\quad
    \{(ss'\triangleright y,s'\triangleright y)\ |\ s\in N,\ s'\in \bG,\ y\in Y\}\subseteq V,
  \end{equation*}
  plainly equivalent to \Cref{eq:eunif} by simply rewriting
  \begin{equation*}
    (ss'\triangleright y,s'\triangleright y)
    =
    (s\triangleright(s'\triangleright y),s'\triangleright y)
    =
    (s\triangleright y',y')
  \end{equation*}
  for $y':=s'\triangleright y$. 
\end{remark}

Constraints one might impose on flows in $\cat{Top}$ or $\cat{Unif}$ or any number of analogues (Hausdorff spaces, etc.) include the following.

\begin{examples}\label{exs:flowconstraints}
  \begin{enumerate}[(1),wide]
  \item\label{item:exs:flowconstraints:top} When $\bM\in \cat{Mon}(\cat{Top})$ and $\cC\subseteq \cat{Top}$ is a subcategory we have the continuous flows therein, making up the category $\tensor[_\iota]{\cC}{^{\bM}}$ of \Cref{not:flowtypes}. 
    
  \item\label{item:exs:flowconstraints:septop} Still assuming $\bM$ topological, there are also the {\it separately} continuous flows mentioned in \Cref{res:postflow}\Cref{item:res:postflow:allconcrete}.    

  \item\label{item:exs:flowconstraints:jointsepmix} It is natural at this stage to regard the two preceding examples as polar extremes along a topological-action-strength axis, with \Cref{item:exs:flowconstraints:top} most and \Cref{item:exs:flowconstraints:septop} least constraining. Mixtures are conceivable: one might consider, for instance, separately continuous actions $\bM\times X\to X$ that are jointly continuous when restricted to a fixed submonoid $\bL\le \bM$. 
    
  \item\label{item:exs:flowconstraints:eunif} Take $\cC=\cat{Unif}$ (or subcategories thereof: $\cat{Unif}_{T_2}$, etc.) and $\bG$ a topological group. We have already recalled in \Cref{eq:eunif} the {\it bounded} (or {\it equiuniform}) flows, constituting the category denoted by $\cat{EUnif}^{\bG}$ in \cite[Definition 3.2(3)]{megr_max-equiv-cpct} and $\tensor[_b]{\cat{Unif}}{^{\bG}}$ in \Cref{eq:unifgdiag}.
    
  \item\label{item:exs:flowconstraints:unif} With $\cC$ and $\bG$ as in \Cref{item:exs:flowconstraints:eunif}, there is the category of {\it $\pi$-uniform} actions (or {\it quasi-bounded} $\bG$-flows in $\cat{Unif}$) of \cite[Definition 3.2(4)]{megr_max-equiv-cpct}, denoted there by $\tensor[]{\cat{Unif}}{^{\bG}}$. The requirement is that for every entourage $W\subseteq X^2$, the action $\bG\times X\to X$ (or rather its Cartesian square) map some
    \begin{equation*}
      \Delta_{N}\times V
      :=
      \left\{(s,s)\ |\ s\in N\right\}
      \times V
      \subseteq
      \bG^2\times X^2
      \cong
      (\bG\times X)^2
    \end{equation*}
    into $W$ for an identity neighborhood $N\ni 1\in \bG$ and an entourage $V\subseteq X^2$ over the entourages of the uniformity on $X$.

  \item As the nomenclature ({\it bounded} vs. {\it quasi-bounded}) suggests, and \cite[Remark 3.3(5) and paragraph preceding Remarks 3.3]{megr_max-equiv-cpct} observe, \Cref{item:exs:flowconstraints:eunif} constrains $\bG$-flows in $\cat{Unif}$ strictly more onerously than \Cref{item:exs:flowconstraints:unif}. In the hybridization spirit of \Cref{item:exs:flowconstraints:jointsepmix}, one could concoct categories of flows quasi-bounded globally and bounded when restricted to subgroups $\bH\le \bG$.
    
  \item\label{item:exs:flowconstraints:internalunif} Mimicking continuous flows in $\cat{Top}$, where the action $\bG\times X\to X$ must be a morphism in that category, appropriate choices of $\cF$ will model as $\tensor[_{\cF}]{\cat{Unif}}{^{\bG}}$ categories of flows for which $\bG\times X\to X$ is uniformly continuous when
    \begin{itemize}[wide]
    \item $X$ is given its original uniformity on both sides;
    \item $\bG$ is given its right or left uniformity, or the bilateral \cite[Definition-Proposition 2.2]{rd_unif-gps} analogue;
    \item and $\bG\times X$ its product uniformity.
    \end{itemize}
    Note, however, that $\bG$ might not be an internal {\it group} in $\cat{Unif}$ for some of these choices: the inversion map interchanges the left and right uniformities, so cannot be uniformly continuous for either of these if the two uniformities do not coincide. 
   
  \item\label{item:exs:flowconstraints:sepunif} Once more as in \Cref{item:exs:flowconstraints:septop}, having equipped $\bM$ with a uniformity, we can recover the category of {\it separately} uniformly continuous flows $\bM\times X\to X$ in $\cat{Unif}$ as a $\tensor[_\cF]{\cC}{^\bM}$.
  \end{enumerate}
\end{examples}

\begin{remark}\label{re:jointsep}
  The distinction drawn in \Cref{exs:flowconstraints}\Cref{item:exs:flowconstraints:septop} between joint and separate continuity does matter in practice. Linear representations $\bG\times E\to E$ of compact Hausdorff topological groups on topological vector spaces, usually \cite[\S 2, p.13]{rob} assumed separately continuous, can easily fail to be jointly continuous (\cite[Example 2.2]{2312.12975v1}, for instance). Joint continuity is, however, automatic \cite[p.VIII.9, Proposition 1]{bourb_int_en_7-9} if the topological vector space $E$ is {\it barreled} \cite[\S 21.2]{k_tvs-1}.
\end{remark}

One can rework much of the above internally to compactly generated spaces. 

\begin{examples}\label{exs:filterstructskr}
  \begin{enumerate}[(1),wide]
  \item\label{item:exs:filterstructskr:ktop} For $\bM\in\cat{Mon}(\cat{Top}_{\kappa})$ there are categories $\tensor[_\iota]{\cC}{^{\bM}}$ of flows continuous for the $\kappa$-product. 
    
    \cite[\S 5.1]{dvr_ttg} writes $\cat{KR}$ for $\cat{Top}_{T_2,\kappa}$ and $\cat{KRGRP}$ for $\cat{Gr}(\cat{Top}_{T_2,\kappa})$. The resulting category of $\bG$-flows discussed here is the $\cat{k-KR}^{\bG}$ of \cite[\S 5.3]{dvr_ttg}.

  \item\label{item:exs:filterstructskr:krsep} As in \Cref{exs:flowconstraints}\Cref{item:exs:flowconstraints:septop}, one can weaken the preceding constraint to separate continuity.

  \item Take $\cC=\cat{Top}_{T_2,\kappa}$ and $\bM\in \cat{Mon}(\cC)$ as in \Cref{exs:filterstructskr}\Cref{item:exs:filterstructskr:ktop}, but {\it strengthen} that constraint (as opposed to weakening it, as \Cref{item:exs:filterstructskr:krsep} does): consider actions $\bM\times X\to X$ jointly continuous for the usual product topology (rather than the finer $\kappa$-product).
    
    Per \Cref{pr:krprodpres} below, this product-structure mixture (ordinary versus $\kappa$) produces a flow category that is less well-behaved for our purposes (monadicity, etc.).
  \end{enumerate}
\end{examples}

\begin{proposition}\label{pr:krprodpres}
  For an internal monoid $\bM\in \cat{Mon}(\cat{Top}_{T_2,\kappa})$ the following conditions are equivalent.
  \begin{enumerate}[(a),wide]
  \item\label{item:pr:krprodpres:lc} $\bM$ is locally compact.

  \item\label{item:pr:krprodpres:mon} The forgetful functor $\tensor*[_c]{\cat{Top}}{^\bM_{T_2,\kappa}}\xrightarrow{U} \cat{Top}_{T_2,\kappa}$ from flows $\bM\times X\to X$ jointly continuous for the {\it Cartesian} product is monadic.

  \item\label{item:pr:krprodpres:radj} $U$ is a right adjoint. 
    
  \item\label{item:pr:krprodpres:cont} $U$ is continuous.

  \item\label{item:pr:krprodpres:prod} $U$ preserves products.

  \item\label{item:pr:krprodpres:2prod} $U$ preserves binary products.

  \item\label{item:pr:krprodpres:spec2prod} $U$ preserves products of the form $(\bM,\text{translation action})\times (X,\text{trivial action})$. 
  \end{enumerate}
\end{proposition}
\begin{proof}
  For \Cref{item:pr:krprodpres:lc} $\Rightarrow$ \Cref{item:pr:krprodpres:mon} recall \cite[Proposition 7.2.9]{brcx_hndbk-2} that plain Cartesian products with locally compact spaces coincide with the corresponding $\kappa$-products. $U$, in that case, will be the forgetful functor associated to the monad $\bM\times -$ (unambiguous product).

  The other downward implications being formal, it remains to settle \Cref{item:pr:krprodpres:spec2prod} $\Rightarrow$ \Cref{item:pr:krprodpres:lc}. Failure of local compactness would imply \cite[Theorem 3.1 and footnote (5)]{zbMATH03279741} the existence of some $X\in \cat{Top}_{T_2,\kappa}$ for which the $\kappa$-product $\bM\times^{\kappa} X$ is strictly finer than the usual product. But then \Cref{exs:le:oplaxact}\Cref{item:exs:le:oplaxact:k} argues via \Cref{le:oplaxact} that the left-hand translation action on $\bM\times^{\kappa} X$ is not continuous on $\bM\times (\bM\times^{\kappa}X)$, negating \Cref{item:pr:krprodpres:spec2prod}. 
\end{proof}

\begin{remark}\label{re:drop.t2}
  The appeals to \cite{zbMATH03279741,brcx_hndbk-2} in the proof of \Cref{pr:krprodpres} both make the Hausdorff property crucial at least to the argument. It would be natural, at this point, to ask to what extent separation can be dropped, and what shape the corresponding (more general) result would take.
\end{remark}

It will be useful to have the monadicity claims made in the Introduction collected together under one heading, with a more or less common argument. Some are certainly in the literature, e.g. \cite[Theorem 3.1.9]{dvr_ttg} or \cite[Theorem 2.3]{dvr-cttg} for $\cC=\cat{Top}$ (which case is simpler than the others because the left adjoint to $\tensor*[_\iota]{\cat{Top}}{^{\bM}}\to \cat{Top}$ is explicitly $\bM\times -$); I have not been able to trace all back to prior work though. 


\begin{theorem}\label{th:monads}
  Let $\bM$ be a monoid and $\cS,\cJ\subseteq 2^{\bM}$ families of subsets thereof. 
  
  The functors $\tensor[_\cF]{\cC}{^{\bM}}\to \cC$ forgetful of actions meeting a constraint $\cF$ are monadic in all of the following cases. 
  \begin{enumerate}[(a),wide]
  \item\label{item:th:monads:top} $\bM$ is a topological monoid, $\cC$ is any of the subcategories $\cat{Top}_{\bullet}$ with
    \begin{equation*}
      \bullet\in \left\{\text{blank},\ T_0,\ T_1,\ T_2,\ T_{2f},\ T_{3\frac 12}\right\}
    \end{equation*}
    or $\cat{Cpct}_{T_2}$, and the actions $\bM\times X\to X$ are required to be separately continuous over $S\times X$, $S\in \cS$ and jointly continuous over $J\times X$, $J\in \cJ$. 
    
  \item\label{item:th:monads:unif} $\bM$ is equipped with a uniformity, $\cC=\cat{Unif}_{\bullet}$ with $\bullet\in\left\{\text{blank},\ T_2,\ (T_2,c)\right\}$, and the actions are again separately (jointly) uniformly continuous over $S\times X$, $S\in \cS$ (respectively $J\times X$, $J\in \cJ$). 

  \item\label{item:th:monads:ktop} $\bM\in\cat{Mon}(\cat{Top}_{\kappa})$, $\cC=\cat{Top}_{\bullet,\kappa}$ with $\bullet\in\{\text{blank},T_2\}$, and the actions are again jointly or separately continuous respectively over
    \begin{equation*}
      \kappa(S)\times X
      \quad\text{and}\quad
      \kappa(J)\times^{\kappa} X
      ,\quad
      S\in \cS,\ J\in \cJ
    \end{equation*}
    where $\kappa$ is the coreflection $\cat{Top}\to \cat{Top}_{\kappa}$ and $\times^{\kappa}$ is the $\kappa$-product of \Cref{se:prel}. 

  \item\label{item:th:monads:mixunif} $\bM$ is a topological group, $\cC$ a category of uniform spaces as in \Cref{item:th:monads:unif} above, and the actions are required to be either bounded (\Cref{exs:flowconstraints}\Cref{item:exs:flowconstraints:eunif}) or quasi-bounded (\Cref{exs:flowconstraints}\Cref{item:exs:flowconstraints:unif}). 
  \end{enumerate}
\end{theorem}
\begin{proof}
  We suppress the left-hand subscript in $\tensor[_\cF]{\cC}{^{\bM}}\to \cC$ to lighten the notation. 
  
  The proof is a standard application of Beck's {\it Precise Tripleability Theorem (PTT)} (\cite[\S 3.3, Theorem 10]{bw} or \cite[Theorem 4.4.4]{brcx_hndbk-2}), whose hypotheses we check in turn (also recalling them in the process).
  
  \begin{enumerate}[(I),wide]
  \item {\bf $\cC^{\bM}\xrightarrow{G} \cC$ is a right adjoint.} As noted above, some cases are simpler than others: for $\cC=\cat{Top}$, for instance, one can simply take the left adjoint of $G$ to be $\bM\times -$ with the left-hand translation action. This move does not apply in general, e.g. for $\cC=\cat{Cpct}_{T_2}$, because $\bM$ is not generally compact Hausdorff. It will thus be cleaner to give a uniform abstract existence argument for the left adjoints $\cC\xrightarrow{F}\cC^{\bM}$ by verifying the conditions of Freyd's {\it Adjoint Functor Theorem} (\cite[Theorem 18.12]{ahs}, \cite[\S V.6]{mcl_2e}).

    First, in all cases, the two categories $\cC^{\bM}$ and $\cC$ are {\it complete} (i.e. \cite[\S V.1]{mcl_2e} have arbitrary small limits) and $G$ is {\it continuous} (meaning \cite[\S V.4]{mcl_2e} that it preserves those limits). It is enough \cite[Theorem 12.3 and Proposition 13.4]{ahs} to check this for products and equalizers; these are defined in all cases set-theoretically via subspace/product topologies and uniformities, the various separation axioms mentioned ($T_2$ and $T_{3\frac 12}$) survive passage to both subspaces and products, and $\bM$-actions simply come along. 
    
    Secondly, the Adjoint Functor Theorem also requires the {\it solution-set condition}: for every object $X\in \cC$ there is a {\it set} (as opposed to a proper class) of morphisms $X\xrightarrow{f_i}G Y_i$ such that every $X\xrightarrow{f}GY$ factors as
    \begin{equation*}
      \begin{tikzpicture}[>=stealth,auto,baseline=(current  bounding  box.center)]
        \path[anchor=base] 
        (0,0) node (l) {$X$}
        +(2,.5) node (u) {$GY_i$}
        +(4,0) node (r) {$GY$.}
        ;
        \draw[->] (l) to[bend left=6] node[pos=.5,auto] {$\scriptstyle f_i$} (u);
        \draw[->] (u) to[bend left=6] node[pos=.5,auto] {$\scriptstyle Gg$} (r);
        \draw[->] (l) to[bend right=6] node[pos=.5,auto,swap] {$\scriptstyle f$} (r);
      \end{tikzpicture}
    \end{equation*}
    This is achievable by taking for the set $\{f_i\}$ all morphisms from $X$ into $\bM$-action carriers of cardinality $\le \kappa$ for some $\kappa$ dependent only on $X$ and $\bM$.

    In those cases where it suffices to factor through $\bM$-invariant subspaces this is obvious: a $\bM$-invariant space generated by the image of a map defined on $X$ has cardinality at most $|\bM|\cdot |X|$. When one has to factor through {\it closed} subspaces, i.e. when a completion process is involved (for $\cC=\cat{Unif}_{T_2,c}$ and $\cat{Cpct}_{T_2}$), recall \cite[\S 2.4]{juh_card} that there is a uniformly-valid bound
    \begin{equation*}
      |Z|\le \exp\exp|D|
      ,\quad
      Z\in\cat{Top}_{T_2}
      ,\quad
      D\subseteq Z\text{ dense}.
    \end{equation*}
    
  \item {\bf $G$ reflects isomorphisms.} This means that a morphism $f$ in the domain $\cC^{\bM}$ of $G$ is an isomorphism provided $Gf$ is one in the codomain $\cC$. The claim is self-evident, as in each case inverses of $\bM$-equivariant maps are again $\bM$-equivariant.
    
  \item {\bf $\cC^{\bM}$ has coequalizers for the pairs $(f,g)$ with $(Gf,Gg)$ contractible and $G$ preserves them.} Recall \cite[\S 3.3, pre Proposition 2]{bw} that a pair $(\varphi_0,\varphi_1)$ of morphisms in a category is {\it contractible (or split)} if it fits into a diagram
    \begin{equation}\label{eq:splitfork}
      \begin{tikzpicture}[>=stealth,auto,baseline=(current  bounding  box.center)]
        \path[anchor=base] 
        (0,0) node (l) {$X'$}
        +(3,0) node (m) {$X$}
        +(5,0) node (r) {$Y$}
        ;
        \draw[->] (l) to[bend left=30] node[pos=.5,auto] {$\scriptstyle \varphi_0$} (m);
        \draw[->] (l) to[bend right=30] node[pos=.5,auto,swap] {$\scriptstyle \varphi_1$} (m);
        \draw[->] (m) to[bend left=10] node[pos=.5,auto] {$\scriptstyle \varphi$} (r);
        \draw[->] (r) to[bend left=10] node[pos=.5,auto] {$\scriptstyle r$} (m);
        \draw[->] (m) to[bend left=0] node[pos=.5,auto,swap] {$\scriptstyle s$} (l);
      \end{tikzpicture}
      ,\quad
      \varphi r=\id
      ,\quad
      \varphi_0 s=\id
      ,\quad
      \varphi_1 s=r\varphi
    \end{equation}
    (whereupon $\varphi$ is automatically \cite[\S 3.3, Proposition 2]{bw} a coequalizer for $(\varphi_0,\varphi_1)$). 

    In all cases $\bM$ acts by $\cC$-isomorphisms, so the action does travel along the coequalizer $\varphi$ of \Cref{eq:splitfork} to give an action $\bM\times Y\to Y$ (with $\varphi_i:=G f_i$, $i=0,1$). The issue is in every case checking that that map is continuous in the appropriate sense (plainly continuous or compatible with the uniformity). This, though, follows from the splitting \Cref{eq:splitfork}: the action in question factors as
    \begin{equation}\label{eq:gy2y}
      \begin{tikzpicture}[>=stealth,auto,baseline=(current  bounding  box.center)]
        \path[anchor=base] 
        (0,0) node (l) {$\bM\times Y$}
        +(2,-.5) node (dl) {$\bM\times X$}
        +(4,-.5) node (dr) {$X$}
        +(6,0) node (r) {$Y$,}
        ;
        \draw[->] (l) to[bend left=6] node[pos=.5,auto] {$\scriptstyle $} (r);
        \draw[->] (l) to[bend right=6] node[pos=.5,auto,swap] {$\scriptstyle \id\times r$} (dl);
        \draw[->] (dl) to[bend right=6] node[pos=.5,auto] {$\scriptstyle $} (dr);
        \draw[->] (dr) to[bend right=6] node[pos=.5,auto,swap] {$\scriptstyle \varphi$} (r);
      \end{tikzpicture}
    \end{equation}
    with the outer bottom arrows morphisms in the desired category $\cC$ and the middle bottom morphism continuous in the requisite sense. 
  \end{enumerate}
\end{proof}


The categories $\cC^{\bM}$ are also presumably well known to be {\it co}complete: see \cite[\S 4.3.3]{dvr_ttg} for $\cC=\cat{Top}$, $\cat{Top}_{T_2}$ and $\cat{Cpct}_{T_2}$ for instance. We record the result in full here, for convenience and uniformity.

\begin{theorem}\label{th:allcocompl}
  The categories $\tensor[_\cF]{\cC}{^{\bM}}\to \cC$ of \Cref{th:monads} are all cocomplete. 
\end{theorem}
\begin{proof}
  Consider a small-domain functor $\cD\xrightarrow{F}\tensor[_\cF]{\cC}{^{\bM}}\to \cC$. A colimit for $F$ is nothing but an initial object in the category $\cat{coc}(F)$ of {\it cocones} over $F$ \cite[Definitions 2.6.5 and 2.6.6]{brcx_hndbk-1}. Because $\tensor[_\cF]{\cC}{^{\bM}}\to \cC$ is continuous between categories $\cat{coc}(F)$ is complete as well, with the forgetful functor $\cat{coc}(F)\to \tensor[_\cF]{\cC}{^{\bM}}\to \cC$ assigning a cocone its tip continuous.

  Freyd's initial-object theorem \cite[\S V.6, Theorem 1]{mcl_2e} will thus ensure the existence of such an initial object assuming, once more, a solution-set condition: a {\it set} $\cS$ of objects in $\cat{coc}(F)$ so that every object receives a morphism from some object in $\cS$. Exactly as in the proof of \Cref{th:monads}, though, it is enough to take for $\cS$ all flows whose carrier space $X$ has cardinality bounded by some cardinal dependent only on $F$ and $\bM$. 
\end{proof}

To return to the issue of equivariant compactifications and monadic lifting:

\begin{corollary}\label{cor:liftreflect}
  For any of the reflective inclusion functors $\cC\lhook\joinrel\to \cD$ the corresponding $\tensor[_\cF]{\cC}{^{\bM}}\lhook\joinrel\to \tensor[_\cF]{\cD}{^{\bM}}$ is also reflective. 
\end{corollary}
\begin{proof}
  As sketched before, in the discussion surrounding \Cref{eq:2commsq}:
  \begin{equation*}
    \begin{tikzpicture}[>=stealth,auto,baseline=(current  bounding  box.center)]
      \path[anchor=base] 
      (0,0) node (l) {$\tensor[_\cF]{\cC}{^{\bM}}$}
      +(2,.5) node (u) {$\tensor[_\cF]{\cD}{^{\bM}}$}
      +(2,-.5) node (d) {$\cC$}
      +(4,0) node (r) {$\cD$}
      ;
      \draw[right hook->] (l) to[bend left=6] node[pos=.5,auto] {$\scriptstyle $} (u);
      \draw[->] (u) to[bend left=6] node[pos=.5,auto] {$\scriptstyle \cat{fgt}$} (r);
      \draw[->] (l) to[bend right=6] node[pos=.5,auto,swap] {$\scriptstyle \cat{fgt}$} (d);
      \draw[right hook->] (d) to[bend right=6] node[pos=.5,auto,swap] {$\scriptstyle $} (r);
    \end{tikzpicture}
  \end{equation*}
  has monadic downward arrows by \Cref{th:monads}, a reflective bottom rightward arrow by assumption and a cocomplete left-hand corner by \Cref{th:allcocompl}. The {\it top} rightward arrow must then also be a right adjoint, by the already-referenced adjunction lifting theorem \cite[Theorem 4.5.6]{brcx_hndbk-2} (which in fact would only have required that $\tensor[_\cF]{\cC}{^{\bM}}$ have coequalizers for {\it reflexive pairs} in the sense of \cite[Exercise 4.8.5]{brcx_hndbk-2}: pairs of morphisms with a common domain and codomain and a common right inverse). 
\end{proof}

As an aside, recall (\cite[p.220]{megr_g-cats}, \cite[\S 4]{MR2166267}) that a $\bG$-flow of a topological group is {\it $\bG$-Tychonoff} if its map to the universal $\bG$-equivariant compactification is a homeomorphism onto its image. Note that the property is one attached to the flow rather than the space: there are, in general, Tychonoff $\bG$-flows that are not $\bG$-Tychonoff; indeed, everything in sight (group and space) can even be metrizable, while the compactification is a singleton (\cite{zbMATH06766536}, building on \cite{zbMATH04086527,megr_g-cats}). For that reason (dependence on the flow rather than its base space), we denote the resulting category by $\left(\tensor*[_\iota]{\cat{Top}}{^{\bG}}\right)_{T_{3\frac 12}}$. The relatively recent \cite{marty_cocompl} proves $\left(\tensor*[_\iota]{\cat{Top}}{^{\bG}}\right)_{T_{3\frac 12}}$ cocomplete for Hausdorff $\bG$ by
\begin{itemize}
\item first constructing coequalizers in the larger category $\tensor*[_\iota]{\cat{Top}}{^{\bG}}_{T_{3\frac 12}}$ \cite[Theorem 1]{marty_cocompl} and then transporting those over to $\left(\tensor*[_\iota]{\cat{Top}}{^{\bG}}\right)_{T_{3\frac 12}}$ \cite[Corollary 2]{marty_cocompl};
\item and also constructing coproducts in that smaller category directly \cite[Theorem 3]{marty_cocompl}. 
\end{itemize}

The cocompleteness result follows from \Cref{th:allcocompl} with no separation constraints on $\bG$, but it might be worth recording the natural intermediate generalization between the two cocompleteness results.

An embedding $\cC\lhook\joinrel\to \cD$ as in \Cref{cor:liftreflect} will single out a special class of objects in the larger category $\tensor[_\cF]{\cD}{^{\bM}}$: the full subcategory
\begin{equation}\label{eq:drelcd}
  \left(\tensor[_\cF]{\cD}{^{\bM}}\right)_{\cC\lhook\joinrel\to \cD}
  \lhook\joinrel\xrightarrow{\quad}
  \tensor[_\cF]{\cD}{^{\bM}}
\end{equation}
on those objects $\bM\times X\to X$ whose reflection $X\to Y$ in $\tensor[_\cF]{\cC}{^{\bM}}$ is an isomorphism (uniform or topological, etc.) onto its image. 

\begin{corollary}\label{cor:cdgengtychcoco}
  For $\cC\lhook\joinrel\to \cD$ as in \Cref{cor:liftreflect} the subcategory \Cref{eq:drelcd} is cocomplete. 
\end{corollary}
\begin{proof}
  An immediate consequence of \Cref{th:allcocompl} and \Cref{cor:liftreflect}, since \Cref{eq:drelcd} is full reflective: the reflection of an object is its image through the reflection in $\tensor[_\cF]{\cC}{^{\bM}}$.
\end{proof}

And returning to the motivating instance:

\begin{corollary}\label{cor:gtychcoco}
  The category $\left(\tensor*[_\iota]{\cat{Top}}{^{\bG}}\right)_{T_{3\frac 12}}$ of continuous $\bG$-Tychonoff flows of a topological group is cocomplete. 
\end{corollary}
\begin{proof}
  This is indeed a particular case of \Cref{cor:cdgengtychcoco}:
  \begin{equation*}
    \left(\tensor*[_\iota]{\cat{Top}}{^{\bG}}\right)_{T_{3\frac 12}}
    \quad
    \subseteq
    \quad
    \tensor*[_{\iota}]{\cat{Top}}{_{T_{3\frac 12}}^{\bG}}
  \end{equation*}
  is precisely \Cref{eq:drelcd} with $\cF:=\iota$, $\cC:=\cat{Cpct}_{T_2}$ and $\cD:=\cat{Top}_{T_{3\frac 12}}$. 
\end{proof}

\begin{remarks}\label{res:coeqpres}
  \begin{enumerate}[(1),wide]
  \item\label{item:res:coeqpres:notcocont} The forgetful functors $\tensor[_\cF]{\cC}{^{\bM}}\to \cC$ of \Cref{th:monads} are more rarely {\it co}continuous:
    
    \begin{itemize}[wide]
    \item For $\cC=\cat{Top}$, for instance, \cite[Theorem 3.4.3]{dvr_ttg} shows that colimits are preserved when the topological group $\bG$ is locally compact Hausdorff, but coequalizers are not preserved generally by \cite[\S 3.4.4]{dvr_ttg}.
      
      The crucial property of $\bG$ in the aforementioned \cite[Theorem 3.4.3]{dvr_ttg} is in fact its exponentiability (i.e. the requirement that $\bG\times -$ be a left adjoint on $\cat{Top}$, as recalled in \Cref{se:prel}). Indeed, this will ensure coequalizer preservation and in fact cocontinuity (the conditions are in fact equivalent: \Cref{pr:gexpntntbl}): $T$ is precisely the monad attached to the monadic functor $\tensor*[_\iota]{\cat{Top}}{^{\bG}}\to \cat{Top}$, and it is a formal exercise to show that in general, given
      \begin{itemize}
      \item a small category $\cD$;
      \item a category $\cC$ admitting $\cD$-shaped colimits $\varinjlim F$, $\cD\xrightarrow{F}\cC$;
      \item and a monad $\cC\xrightarrow{T}\cC$,
      \end{itemize}
      $\cD$-shaped colimits exist in the Eilenberg-Moore category $\cC^T$ and are preserved by the forgetful functor $\cC^{T}\xrightarrow{G} \cC$ if and only if they are preserved by $T$.
      
      The implication ($\Leftarrow$) (also noted in passing in \cite[proof of Lemma 5.5]{zbMATH06011381}) is \cite[Proposition 4.3.2]{brcx_hndbk-2}. Conversely, recall that $T$ can be recovered as
      \begin{equation*}
        T=G\circ\left(\text{left adjoint of G}\right).
      \end{equation*}
      Said left adjoint of course preserves arbitrary colimits, so any colimits preserved by $G$ are preserved by $T$ also. 

      Locally compact spaces (separated or not) are exponentiable \cite[Proposition 7.1.5]{brcx_hndbk-2}, so \cite[Theorem 3.4.3]{dvr_ttg} in fact goes through for possibly non-$T_2$ locally compact groups. See also \cite[Theorem II-4.12]{ghklms_cont-lat} for alternative characterizations of exponentiable spaces. Exponentiability {\it is} equivalent to local compactness assuming $T_2$ (or more generally \cite[Theorem V-5.6]{ghklms_cont-lat}, for {\it sober} spaces, i.e. \cite[Definition O-5.6]{ghklms_cont-lat} those for which irreducible closed sets are closures of unique singletons).
      
    \item For $\cC=\cat{Cpct}_{T_2}$ and locally compact Hausdorff $\bG$ the preservation of coproducts by $\tensor[_\iota]{\cC}{^{\bG}}\to \cC$ is equivalent to the {\it discreteness} of $\bG$ \cite[Theorem 3.1]{chirvasitu2023quantum}. 
            
    \end{itemize}
    
  \item\label{item:res:coeqpres:noctt} Item \Cref{item:res:coeqpres:notcocont} above also shows that in proving monadicity, one could not employ some of the ``coarser'' versions of Beck's theorem. The {\it Reflexive Tripleability Theorem (RTT)} of \cite[Proposition 5.5.8]{rhl_ct-ctxt}, for instance, would require the preservation by $\tensor[_\cF]{\cC}{^{\bM}}\to \cC$ of {\it reflexive} coequalizers, i.e. \cite[\S 3.3, p.108]{bw} coequalizers of pairs $(f,g)$ of parallel morphisms that have a common right inverse. In all cases under consideration, though, that would amount to preservation of {\it arbitrary} coequalizers (which we know does not obtain universally): because coproducts are, space-wise, simply disjoint unions, an arbitrary parallel pair
    \begin{equation*}
      \begin{tikzpicture}[>=stealth,auto,baseline=(current  bounding  box.center)]
        \path[anchor=base] 
        (0,0) node (l) {$X$}
        +(4,0) node (r) {$Y$}
        ;
        \draw[->] (l) to[bend left=6] node[pos=.5,auto] {$\scriptstyle f$} (r);
        \draw[->] (l) to[bend right=6] node[pos=.5,auto,swap] {$\scriptstyle g$} (r);
      \end{tikzpicture}
    \end{equation*}
    expands into a reflexive pair
    \begin{equation*}
      \begin{tikzpicture}[>=stealth,auto,baseline=(current  bounding  box.center)]
        \path[anchor=base] 
        (0,0) node (l) {$X\coprod Y$}
        +(4,0) node (r) {$Y$}
        ;
        \draw[->] (l) to[bend left=20] node[pos=.5,auto] {$\scriptstyle (f,\id_Y)$} (r);
        \draw[->] (l) to[bend right=20] node[pos=.5,auto,swap] {$\scriptstyle (g,\id_Y)$} (r);
        \draw[left hook->] (r) to[bend right=0] node[pos=.5,auto,swap] {$\scriptstyle $} (l);
      \end{tikzpicture}
    \end{equation*}
    with the same coequalizer. 
  \end{enumerate}
\end{remarks}

Incidentally, the argument in \cite[\S 3.4.4]{dvr_ttg} showing (via \cite[Example 1.5.11]{dvr_ttg}) that $\cat{Top}^{(\bQ,+)}\to \cat{Top}$ fails to preserve coequalizers can be amplified into a {\it characterization} of those groups for which such pathologies do not obtain. See also \cite[\S\S 6.3.6-7]{dvr_ttg} for explicit mention and discussion of the comonadicity of $\tensor*[_\iota]{\cat{Top}}{^{\bG}}\to \cat{Top}$ for locally compact Hausdorff $\bG$. 

\begin{proposition}\label{pr:gexpntntbl}
  For a topological group $\bG$, the following conditions are equivalent.
  \begin{enumerate}[(a),wide]
  \item\label{item:pr:gexpntntbl:a} $\bG$ is exponentiable as a topological space, i.e. $\cat{Top}\xrightarrow{\bG\times -}\cat{Top}$ is a left adjoint.
  \item\label{item:pr:gexpntntbl:b} $\bG\times -$ is cocontinuous.
  \item\label{item:pr:gexpntntbl:c} $\bG\times -$ preserves $\cat{Top}$-coequalizers.
  \item\label{item:pr:gexpntntbl:d} $\bG\times -$ preserves quotients in $\cat{Top}$ by equivalence relations.
  \item\label{item:pr:gexpntntbl:comon} $\tensor*[_\iota]{\cat{Top}}{^{\bG}}\xrightarrow{G} \cat{Top}$ is comonadic.
  \item\label{item:pr:gexpntntbl:ladj} $\tensor*[_\iota]{\cat{Top}}{^{\bG}}\xrightarrow{G} \cat{Top}$ is a left adjoint.
  \item\label{item:pr:gexpntntbl:cocont} $\tensor*[_\iota]{\cat{Top}}{^{\bG}}\xrightarrow{G} \cat{Top}$ is cocontinuous.
  \item\label{item:pr:gexpntntbl:coeq} $\tensor*[_\iota]{\cat{Top}}{^{\bG}}\xrightarrow{G} \cat{Top}$ preserves coequalizers.
  \end{enumerate}  
\end{proposition}
\begin{proof}
  The downward implications from \Cref{item:pr:gexpntntbl:a} to \Cref{item:pr:gexpntntbl:d} are obvious (for arbitrary spaces; the group structure is irrelevant), and \cite[Theorem II-4.12]{ghklms_cont-lat} proves \Cref{item:pr:gexpntntbl:a} $\iff$ \Cref{item:pr:gexpntntbl:d} (for $T_0$ spaces, but that assumption is not crucial). The first four conditions are thus equivalent. 
  
  We also plainly have
  \begin{equation}\label{eq:comoniffcoeq}
    \Cref{item:pr:gexpntntbl:comon}
    \xRightarrow{\quad}
    \Cref{item:pr:gexpntntbl:ladj}
    \xRightarrow{\quad}
    \Cref{item:pr:gexpntntbl:cocont}
    \xRightarrow{\quad}
    \Cref{item:pr:gexpntntbl:coeq}.
  \end{equation}
  The first implication reverses by any number of comonadicity, dually to \Cref{th:monads}, because we already know that $G$ is continuous (so preserves all equalizers). Because $G$ preserves coproducts, it is cocontinuous precisely when it preserves coequalizers \cite[Proposition 13.4]{ahs}, so the third implication in \Cref{eq:comoniffcoeq} also backtracks. As for the converse to the {\it second} implication, it is a consequence of the adjoint functor theorem \cite[Theorem 18.12]{ahs} provided we verify the solution-set condition. To that end, note that every map $G(Y)\xrightarrow{f}X$ in $\cat{Top}$ factors (just plain set-theoretically) as
  \begin{equation*}
    \begin{tikzpicture}[>=stealth,auto,baseline=(current  bounding  box.center)]
      \path[anchor=base] 
      (0,0) node (l) {$G(Y)$}
      +(2,.5) node (u) {$X^{\bG}$}
      +(4,0) node (r) {$X$}
      ;
      \draw[->] (l) to[bend left=6] node[pos=.5,auto] {$\scriptstyle $} (u);
      \draw[->] (u) to[bend left=6] node[pos=.5,auto] {$\scriptstyle \text{projection at }1\in \bG$} (r);
      \draw[->] (l) to[bend right=6] node[pos=.5,auto,swap] {$\scriptstyle f$} (r);
    \end{tikzpicture}
  \end{equation*}
  through the $\bG$-equivariant upper left-hand map
  \begin{equation*}
    G(Y)\ni y
    \xmapsto{\quad}
    \big(\bG\ni s\xmapsto{\quad} f(s\triangleright y)\in X\big)
    \in X^{\bG}
  \end{equation*}
  (with $X^\bG$ acted upon by $\bG$ via $s\triangleright \varphi:=\varphi(- \cdot s)$). $f$ will thus factor through $G(\pi)$ for some quotient $G\xrightarrow{\pi}\overline{G}$ in $\tensor*[_\iota]{\cat{Top}}{^{\bG}}$ of cardinality $|\overline{G}|\le |X^{\bG}|$ (so bounded independently of $G$).
  
  We now have
  \begin{equation*}
    \Cref{item:pr:gexpntntbl:d}
    \xLeftrightarrow{\ }
    \Cref{item:pr:gexpntntbl:c}
    \xLeftrightarrow{\ }
    \Cref{item:pr:gexpntntbl:b}
    \xLeftrightarrow{\ }
    \Cref{item:pr:gexpntntbl:a}    
    \xLeftrightarrow{\text{\Cref{res:coeqpres}\Cref{item:res:coeqpres:notcocont}}}
    \Cref{item:pr:gexpntntbl:coeq}
    \xLeftrightarrow{\ }
    \Cref{item:pr:gexpntntbl:cocont}
    \xLeftrightarrow{\ }
    \Cref{item:pr:gexpntntbl:ladj}
    \xLeftrightarrow{\ }
    \Cref{item:pr:gexpntntbl:comon},
  \end{equation*}
  and we are done.  
\end{proof}

\begin{remark}\label{re:hownotpresquot}
  It might also be instructive to adapt \cite[Example 1.5.11 and \S 3.4.4]{dvr_ttg} (there specific to $\bG=\bQ$) to see directly how coequalizer preservation fails whenever $\bG\times -$ fails to preserve a quotient $\overline{X}:=X/R$ in $\cat{Top}$ by an equivalence relation $R\subseteq X^2$: the implication \Cref{item:pr:gexpntntbl:coeq} $\Rightarrow$ \Cref{item:pr:gexpntntbl:d} of \Cref{th:allcocompl}, in other words. 
    
  Failure of quotient preservation means that the quotient topology  $\bG\times^{q} \overline{X}$ is strictly finer than the usual product topology $\bG\times^{\pi} \overline{X}$. It then follows that the translation action
  \begin{equation*}
    \bG\times^{\pi} \left(\bG\times^q \overline{X}\right)
    \xrightarrow{\quad}
    \bG\times^q \overline{X}
  \end{equation*}
  cannot be continuous. The general phenomenon driving this remark is recorded in \Cref{le:oplaxact} below.
\end{remark}

The technical principle noted in \Cref{le:oplaxact} below is certainly a simple one, but worth isolating: it has already surreptitiously come in handy (at least) twice.

Recall \cite[Definition 1.2.10]{leinst_hc_2004} that a {\it colax} (sometimes {\it oplax} \cite[p.271]{zbMATH01903518}) monoidal functor $\cD\xrightarrow{F}\cC$ between monoidal categories $(\cD,\otimes,1)$ and $(\cC,\otimes,1)$ is one equipped with morphisms
\begin{equation*}
  F(\bullet\otimes -)
  \xrightarrow{\quad\phi_{\bullet,-}\quad}
  F\bullet\otimes F-
  \quad\text{and}\quad
  1\xrightarrow{\quad\phi\quad}F1,
\end{equation*}
natural and compatible with the associativity and unitality constraints in the guessable sense. 

\begin{lemma}\label{le:oplaxact}
  Let $\cD\xrightarrow{F}\cC$ be an op-lax monoidal functor,
  \begin{equation*}
    \bM\otimes \bM
    \xrightarrow{\quad\mu\quad}
    \bM
    ,\quad
    1
    \xrightarrow{\quad\eta\quad}
    \bM
    \quad
    \text{in}
    \quad
    \cD
  \end{equation*}
  an internal monoid and $X\in \cD$ an object.

  If there is a factorization 
  \begin{equation}\label{eq:le:oplaxact:fact}
    \begin{tikzpicture}[>=stealth,auto,baseline=(current  bounding  box.center)]
      \path[anchor=base] 
      (0,0) node (l) {$F(\bM\otimes \bM\otimes X)$}
      +(4,.5) node (u) {$F(\bM)\otimes F(\bM\otimes X)$}
      +(7,0) node (r) {$F(\bM\otimes X)$}
      ;
      \draw[->] (l) to[bend left=6] node[pos=.5,auto] {$\scriptstyle \phi_{\bM,\bM\otimes X}$} (u);
      \draw[->] (u) to[bend left=6] node[pos=.5,auto] {$\scriptstyle $} (r);
      \draw[->] (l) to[bend right=6] node[pos=.5,auto,swap] {$\scriptstyle F(\mu\otimes \id_X)$} (r);
    \end{tikzpicture}
  \end{equation}
  then $F(\bM\otimes X)\xrightarrow{\phi_{\bM,X}}F\bM\otimes FX$ has a left inverse. In particular, if $\phi_{\bM,X}$ is epic then it is an isomorphism. 
\end{lemma}
\begin{proof}
  Left-invertible epimorphisms are isomorphisms by (the dual to) \cite[Proposition 7.36]{ahs}, hence the second claim given the first. For the latter, fit a factorization \Cref{eq:le:oplaxact:fact} into the commutative
  \begin{equation*}
    \begin{tikzpicture}[>=stealth,auto,baseline=(current  bounding  box.center)]
      \path[anchor=base] 
      (0,0) node (l) {$F(\bM\otimes \bM\otimes X)$}
      +(4,.5) node (u) {$F(\bM)\otimes F(\bM\otimes X)$}
      +(7,0) node (r) {$F(\bM\otimes X)$:}
      +(-4,.5) node (ul) {$F(\bM\otimes X)$}
      +(1,1.5) node (ur) {$F\bM\otimes FX$}
      ;
      \draw[->] (l) to[bend left=6] node[pos=.5,auto] {$\scriptstyle \phi_{\bM,\bM\otimes X}$} (u);
      \draw[->] (u) to[bend left=6] node[pos=.5,auto] {$\scriptstyle $} (r);
      \draw[->] (l) to[bend right=6] node[pos=.5,auto,swap] {$\scriptstyle F(\mu\otimes \id_X)$} (r);
      \draw[->] (ul) to[bend right=6] node[pos=.5,auto,swap] {$\scriptstyle F(\id_{\bM}\otimes\eta\otimes \id_X)$} (l);
      \draw[->] (ul) to[bend left=6] node[pos=.5,auto] {$\scriptstyle \phi_{\bM,X}$} (ur);
      \draw[->] (ur) to[bend left=6] node[pos=.5,auto] {$\scriptstyle \id_{F\bM}\otimes F(\eta\otimes\id_X)$} (u);
      \draw[->] (ul) .. controls (-6,-1.5) and (6,-.8) .. node[pos=.5,auto,swap] {$\scriptstyle \id$} (r);
    \end{tikzpicture}
  \end{equation*}
  the triangle commutes by assumption and the upper left-hand square by the naturality of $\phi_{\bullet,-}$. We have the requisite left inverse to $\phi_{\bM,X}$ in the composition of the two upper right-hand maps. 
\end{proof}

\begin{examples}\label{exs:le:oplaxact}
  The two occasions for applying \Cref{le:oplaxact} alluded to above are as follows.
  \begin{enumerate}[(1),wide]
  \item\label{item:exs:le:oplaxact:quot} Take $\cC=\cat{Top}$ with its usual Cartesian monoidal structure (or any number of satellite variations; $\cat{Top}_{T_2}$, etc.). $\cD$ is the category of equivalence relations
    \begin{equation*}
      (X,R)
      ,\quad
      R\subseteq X\times X
      ,\quad
      X\in \cC,
    \end{equation*}
    again with the Cartesian structure. The functor $\cD\xrightarrow{F}\cC$ is
    \begin{equation*}
      \cD\ni (X,R)
      \xmapsto{\quad F\quad}
      \overline{X}:=X/R\in \cC.
    \end{equation*}
    The colax structure derives from the familiar observation that quotients of products have at least as fine a topology as the corresponding products of quotients. The canonical $\phi_{\bullet,-}$ are also plainly epic, being set-theoretic bijections.
    
    Identify $\bM\in \cat{Mon}(\cC)$ with its diagonal equivalence relation, so that it becomes a monoid in $\cD$ as well. \Cref{le:oplaxact} then applies, and says that {\it whenever} $\bM\times -$ fails to preserve a quotient $\overline{X}:=X/R$ the left-hand translation $\bM$-action on the quotient space $\bM\times^{q} \overline{X}$ with the quotient topology fails to be continuous for the {\it Cartesian} topology on $\bM\times(\bM\times^q \overline{X})$. 

  \item\label{item:exs:le:oplaxact:k} Take for $F$ the full embedding
    \begin{equation*}
      \cD
      :=
      \cat{Top}_{T_2,\kappa}
      \lhook\joinrel\xrightarrow[]{\quad F\quad}
      \cat{Top}_{T_2}
      =:\cC
    \end{equation*}
    of the category of Hausdorff {\it compactly generated} spaces, i.e. \cite[Definition 43.8]{wil_top} those $X\in\cat{Top}_{T_2}$ whose open sets are precisely those whose intersection with every compact subspace is open (equivalently: $X$ carries the {\it final topology} \cite[\S I.2.4, Proposition 6]{bourb_top_en_1} induced by the inclusions of its compact subspaces).
    
    $\cat{Top}_{T_2,\kappa}$ is coreflective in $\cat{Top}_{T_2}$ \cite[\S VII.8, Proposition 2]{mcl_2e}, so in particular (co)complete. It is also Cartesian closed \cite[\S VII.8, Theorem 3]{mcl_2e} for its product $\times^{\kappa}$ (henceforth the {\it $\kappa$-product})  obtained by composing the usual Cartesian product with the coreflection $\cat{Top}_{T_2}\to \cat{Top}_{T_2,\kappa}$; this gives the colaxity
    \begin{equation*}
      (\bullet)\times^{\kappa} (-)
      \xrightarrow{\quad\phi_{\bullet,-}}
      (\bullet)\times (-)
    \end{equation*}
    required by \Cref{le:oplaxact}, again epic because bijective. Per that result, we will have $\bM\in \cat{Mon}(\cat{Top}_{T_2,\kappa})$ failing to operate plain-$\times$-continuously on $\bM\times^{\kappa}X$ whenever the latter carries a strictly finer topology than $\bM\times X$. This phenomenon is what drove the proof of \Cref{pr:krprodpres}.
  \end{enumerate}
\end{examples}

There are also functors linking the categories $\tensor[_\cF]{\cC}{^{\bM}}$ of \Cref{th:monads} for fixed $\cC$ and $\bM$ varying $\cF$: one can strengthen the constraint $\cF$ to $\cF'$ (notation: $\cF\preceq \cF'$) in the sense of making it more demanding. Examples include
\begin{itemize}[wide]
\item enlarging $\cS$ and/or $\cJ$;

\item enlarging individual sets belonging to $\cS$ and/or $\cJ$;

\item strengthening the quasi-boundedness of \Cref{exs:flowconstraints}\Cref{item:exs:flowconstraints:unif} to the boundedness of \Cref{exs:flowconstraints}\Cref{item:exs:flowconstraints:eunif};

\item and in turn strengthening the latter to joint uniform continuity for actions $\bM\times X\to X$ as in \Cref{exs:flowconstraints}\Cref{item:exs:flowconstraints:internalunif}, upon equipping $\bM$ with any of its standard uniformities (left, right, bilateral). 
\end{itemize}

Any such relation $\cF\preceq \cF'$ produces a full inclusion functor $\tensor[_{\cF'}]{\cC}{^{\bM}}\subseteq \tensor[_{\cF}]{\cC}{^{\bM}}$. It is at this point not surprising, perhaps, that those inclusions reflect (i.e. admit left adjoints).

\begin{theorem}\label{th:strengthreflect}
  For all listed instances of constraint strengthening $\cF\preceq \cF'$ in the context of \Cref{th:monads} the resulting inclusion $\tensor[_{\cF'}]{\cC}{^{\bM}}\subseteq \tensor[_{\cF}]{\cC}{^{\bM}}$ is reflective. 
\end{theorem}
\begin{proof}
  By \Cref{th:monads} the categories are complete and the inclusion is continuous because it fits into a commutative triangle
  \begin{equation*}
    \begin{tikzpicture}[>=stealth,auto,baseline=(current  bounding  box.center)]
      \path[anchor=base] 
      (0,0) node (l) {$\tensor[_{\cF'}]{\cC}{^{\bM}}$}
      +(2,-.5) node (d) {$\cC$}
      +(4,0) node (r) {$\tensor[_{\cF}]{\cC}{^{\bM}}$.}
      ;
      \draw[right hook->] (l) to[bend left=6] node[pos=.5,auto] {$\scriptstyle \iota$} (r);
      \draw[->] (r) to[bend left=6] node[pos=.5,auto] {$\scriptstyle \text{monadic}$} (d);
      \draw[->] (l) to[bend right=6] node[pos=.5,auto,swap] {$\scriptstyle \text{monadic}$} (d);
    \end{tikzpicture}
  \end{equation*}
  The conclusion now follows from the adjoint functor theorem \cite[Theorem 18.12]{ahs} after again observing that the solution-set condition is satisfied in all cases: a morphism $X\to \iota Y$ in the larger category factors through $\iota Y'$ with the cardinality of $Y'$ bounded uniformly in terms of only $X$ and the fixed data $\cC$, $\bM$, $\cF$ and $\cF'$. 
\end{proof}

\begin{remark}\label{re:qb2brefl}
  The particular case
  \begin{equation*}
    \text{
      quasi-boundedness
      (\Cref{exs:flowconstraints}\Cref{item:exs:flowconstraints:unif})
    }\ 
    =:
    \ 
    \cF\preceq\cF'
    \ 
    :=
    \ 
    \text{
      boundedness
      (\Cref{exs:flowconstraints}\Cref{item:exs:flowconstraints:eunif})
    }
  \end{equation*}
  of \Cref{th:strengthreflect} is the construction $(X,\cU)\mapsto (X,\cU^{\bG})$ of \cite[Lemma 3.8]{megr_max-equiv-cpct}, attaching a uniform space carrying a bounded $\bG$-action to one carrying only a quasi-bounded one.

  \cite[Lemma 3.8]{megr_max-equiv-cpct} does not phrase the construction in terms of universality, but the check that that universality does obtain is simple enough: the entourages of the original uniformity $\cU$ are enlarged by fiat into those of $\cU^{\bG}$ so as to render the original action bounded (hence a uniformly continuous map $(X,\cU)\to (X,\cU^{\bG})$), and the enlargement is plainly optimal subject to this boundedness constraint.
\end{remark}




\addcontentsline{toc}{section}{References}

\Addresses

\end{document}